\documentclass[reqno]{amsart}

\usepackage{microtype}
\usepackage{hyperref}
         
\usepackage{amsmath}             
\usepackage{amsfonts}             
\usepackage{amsthm}               
\usepackage{amsbsy}
\usepackage{amssymb}
\usepackage{amsthm}
\usepackage{amsrefs}
\usepackage{multirow}

\usepackage{enumerate}

\usepackage{tikz}
\usetikzlibrary{shapes,snakes,calc,arrows}
\usetikzlibrary{decorations.pathreplacing,shapes.geometric}
\usetikzlibrary{calc,positioning}
\tikzset{Box/.style={very thick, rounded corners}}
\tikzset{marked/.style={star, star point height = .75mm, star points =5, fill=black,minimum size=2mm, inner sep=0mm} }
\tikzset{verythickline/.style = {line width=7pt}}
\tikzset{thickline/.style = {line width=5pt}}
\tikzset{medthick/.style = {line width=3pt}}
\tikzset{med/.style = {line width=2pt}}
\tikzset{count/.style = {fill=white,circle,draw,thin, inner sep=2pt}}
\tikzset{rcount/.style = {fill=white,rectangle,draw,thin,inner sep=2pt, rounded corners}}
\tikzset{cpr/.style = {draw,fill=white,rectangle,thin, rounded corners}}

\newtheorem{thm}{Theorem}[section]

\theoremstyle{definition}
\newtheorem{defn}[thm]{Definition}

\newtheorem{exam}[thm]{Example}

\newcommand{\bC}{{\mathbb{C}}}
\newcommand{\bN}{{\mathbb{N}}}
\newcommand{\bR}{{\mathbb{R}}}
\newcommand{\bT}{{\mathbb{T}}}
\newcommand{\bZ}{{\mathbb{Z}}}

\newcommand{\A}{{\mathcal{A}}}
\newcommand{\B}{{\mathcal{B}}}

\newcommand{\F}{{\mathcal{F}}}
\renewcommand{\H}{{\mathcal{H}}}

\renewcommand{\phi}{\varphi}

\newcommand{\fM}{{\mathfrak{M}}}
\newcommand{\fN}{{\mathfrak{N}}}

\newcommand{\alg}{\mathrm{alg}}
\newcommand{\qand}{\quad\text{and}\quad}

\author{Paul Skoufranis}
\address{Department of Mathematics and Statistics, York University, 4700 Keele Street, Toronto, Ontario, M3J 1P3, Canada}
\email{pskoufra@yorku.ca}
\date{\today}
\thanks{This research was completed with the support of NSERC (Canada) grant RGPIN-2017-05711.}

\begin{document}

\nocite{*}

\title{Non-Commutative Stochastic Processes and Bi-Free Probability}

\begin{abstract}
In this paper, a connection between bi-free probability and the theory of non-commutative stochastic processes is examined.  Specifically it is demonstrated that the transition operators for non-commutative stochastic processes can be modelled using technology from bi-free probability.  Several important examples are recovered with this approach and new formula are obtained for processes with free increments.  The benefits of this approach are also discussed.
\end{abstract}

\maketitle

\section{Introduction}

Non-commutative stochastic processes are continua of operators sitting inside non-commutative probability spaces.  Such processes can be used to model several phenomena in free probability and quantum information theory (see \cites{B1998, BKS1997}).  More recently, Voiculescu in \cite{V2014} introduced the notion of bi-free independence to simultaneously model the left and right actions of operators on reduced free product space.  Much of the terminology used in \cite{V2014} was based on the idea that the left action would model the past, the right action would model the future, and the theory of bi-free probability would act as a transition between the past and the future.  The goal of this paper is to realize this idea by using bi-free probability to model the transition between different times in non-commutative stochastic processes.

After reviewing the basics and transformations of bi-free probability in Section \ref{sec:bi-free} and some terminology on non-commutative stochastic processes in Section \ref{sec:stochastic}, Sections \ref{sec:self-adjoint} and \ref{sec:unitary} will examine the relation between bi-free probability and transition operators for self-adjoint and unitary non-commutative stochastic processes respectively.  Transition operators play a fundamental role in determining the progression from one point in a non-commutative stochastic process to another and dictate the expectations of the future onto the past.  It will be demonstrated that the transition operators of free Gaussian Markov processes are modelled by the bi-free Gaussian distributions (Example \ref{exam:bi-free-central-limit}) and free Poisson processes are modelled by bi-free Poisson distributions (Example \ref{exam:poisson}).  New formula to determine the transition operators for non-commutative stochastic process with freely independent increments will be examined (Theorems \ref{thm:free-additive-increments-Cauchy-transform-formula} and \ref{thm:free-multiplicative-increments-Cauchy-transform-formula}) and applied recovering some common examples.  Furthermore, the role of bi-free probability in studying non-commutative stochastic processes where the increments are not freely independent will be examined and a ``central limit'' type theorem will be developed for adding non-commutative stochastic processes (Theorem \ref{thm:central}).

\section{Bi-Free Probability Background}
\label{sec:bi-free}

In this section, we briefly recall some of the basic concepts in bi-free probability theory from \cites{CNS2015-2, GHM2016, HHW2018, HW2016, HW2018, S2016-2, S2018, V2014, V2016-2, BBGS2018, S2016-1, CNS2015-1, V2016-2}.

\begin{defn}[\cite{V2014}]
Let $(\A, \varphi)$ be a non-commutative probability space (i.e. $\A$ a unital C$^*$-algebra and $\varphi : \A \to \bC$ unital positive linear functional).  Pairs of unital $*$-subalgebras $(A_{\ell,1}, A_{r, 1})$ and $(A_{\ell,2}, A_{r, 2})$ of $\A$ are said to be \emph{bi-freely independent with respect to $\varphi$} if for $k \in \{1,2\}$ there exist Hilbert spaces $\H_k$, unit vectors $\xi_k \in \H_k$, and unital *-homomorphisms $\alpha_k : A_{\ell, k} \to \mathcal{B}(\H_k)$ and $\beta_k : A_{r,k} \to \mathcal{B}(\H_k)$ such that the following diagram commutes
\[
\begin{tikzpicture}[baseline]
\tikzset{
    myarrow/.style={->, >=latex', shorten >=1pt, thick},
} 
			\node[above] at (0,1.5) {$A_{\ell,1} \ast A_{r,1} \ast A_{\ell,2} \ast A_{r,2}$};
			\node[above] at (4.5,1.5) {$\A$};
			\node[above] at (7,1.5) {$\bC$};
			\node[below] at (0,0) {$\mathcal{B}(\H_1) \ast \mathcal{B}(\H_1) \ast \mathcal{B}(\H_2) \ast \mathcal{B}(\H_2)$};
			\node[below] at (7,0) {$\mathcal{B}(\H_1 \ast \H_2)$};
			\draw[myarrow] (0, 1.5) -- (0,0);
			\node[right] at (0,.75) {$\alpha_1 \ast \beta_1 \ast \alpha_2 \ast \beta_2$};
			\draw[myarrow] (2, 1.8) -- (3.95,1.8);
			\node[above] at (3,1.8) {$i$};
			\draw[myarrow] (2.75, -.3) -- (6,-.3);
			\node[above] at (4.35,-.3) {$\lambda_1 \ast \rho_1 \ast \lambda_2 \ast \rho_2$};
			\draw[myarrow] (5.05, 1.8) -- (6.75,1.8);
			\node[above] at (5.75, 1.8) {$\varphi$};
			\draw[myarrow] (7, 0) -- (7,1.5);
			\node[left] at (7, .75) {$\varphi_*$};
		    \end{tikzpicture}
\]
where $i$ is the inclusion map, $\H_1 \ast \H_2$ is the reduced free product of Hilbert spaces with respect to the vectors $\xi_1$ and $\xi_2$, $\varphi_*$ is the vector state on $\H_1 \ast \H_2$ induced by the vacuum vector, $\lambda_k$ is the left action of $\H_k$ on $\H_1 \ast \H_2$, $\rho_k$ is the right action of $\H_k$ on $\H_1 \ast \H_2$, and the free product of *-algebras are as normal.

A pair $(X, Y)$ is used to denote two operators $X, Y \in \A$ where $X$ is viewed as a left operator and $Y$ is viewed as a right operator in terms of the above definition.  Thus two pairs $(X_1, Y_1)$ and $(X_2, Y_2)$ are said to be bi-freely independent if $(*\text{-}\alg(X_1), *\text{-}\alg(Y_1))$ and $(*\text{-}\alg(X_2), *\text{-}\alg(Y_2))$ are bi-freely independent.
\end{defn}

It is elementary to see that if $(X_1, Y_1)$ and $(X_2, Y_2)$ are bi-freely independent with respect to $\varphi$, then $X_1$ and $X_2$ are freely independent with respect to  $\varphi$, $Y_1$ and $Y_2$ are freely independent with respect to $\varphi$, $X_1$ and $Y_2$ are classically independent with respect to $\varphi$ (i.e. $X_1$ and $Y_2$ commute in distribution and $\varphi(X_1^n Y_2^m) = \varphi(X_1^n) \varphi(Y_2^m)$ for all $n,m \in \bN$), and $X_2$ and $Y_1$ are classically independent with respect to $\varphi$.  The converse does not hold (see \cite{S2016-2}).  However, in the commutative case, it is possible to deduce bi-free independence from free independence in certain situations.

\begin{thm}[\cite{CNS2015-1}*{Theorem 10.2.1}]\label{thm:bi-free-from-free}
Let $\H$ be a Hilbert space, let $\xi \in \H$ be a unit vector, and let $\varphi$ be the the vector space corresponding to $\xi$.  Suppose $(A_{\ell,1}, A_{r, 1})$ and $(A_{\ell,2}, A_{r, 2})$ are pairs of unital $*$-algebras in $\B(\H)$ such that
\begin{enumerate}
\item $A_{\ell, n}$ and $A_{r,m}$ commute for all $n, m \in \{1,2\}$, and
\item for each $T \in A_{r,k}$ there exists an $S \in A_{\ell,k}$ such that $T \xi = S \xi$ for all $k \in \{1,2\}$.
\end{enumerate}
Then $(A_{\ell,1}, A_{r, 1})$ and $(A_{\ell,2}, A_{r, 2})$ are bi-freely independent with respect to $\varphi$ if and only if $A_1$ and $A_2$ are freely independent with respect to $\varphi$.
\end{thm}

Similar to the theory of free cumulants via non-crossing partitions from \cite{S1994}, \cite{V2016-2} introduced the bi-free cumulants and \cite{CNS2015-2} extended the notion of bi-free cumulants using bi-non-crossing partitions.  As this paper will only make use of some elementary facts pertaining to the bi-free cumulants, we highlight that which is required here.

Given a pair $(Z_\ell, Z_r)$ in a non-commutative probability space $(\A, \varphi)$, for all $n \in \bN$ and maps $\chi : \{1,\ldots, n\} \to \{\ell, r\}$, there is a bi-free cumulant corresponding to $\chi$ denoted
\[
\kappa_\chi(Z_{\chi(1)}, \ldots, Z_{\chi(n)}).
\]
In the case that $Z_\ell$ and $Z_r$ commute, all joint moments of $(Z_\ell, Z_r)$ can be reduced to a moment of the form $\varphi(Z_\ell^n Z_r^m)$ for $n,m \in \bN$ and thus \cite{CNS2015-2} implies only cumulants for $\chi : \{1,\ldots, n+m\} \to \{\ell, r\}$ where there exists a $k$ such that $\chi(k) = \ell$ if $k \leq n$ and $\chi(k) = r$ if $k > n$ matter in the moment-cumulant formula.  Consequently, for such $\chi$, we simplify notation to
\[
\kappa_\chi(Z_{\chi(1)}, \ldots, Z_{\chi(n)}) = \kappa_{n,m}(\underbrace{Z_{\ell}, \ldots, Z_{\ell}}_{n \text{ times}}, \underbrace{Z_r, \ldots, Z_r}_{m \text{ times}}) = \kappa_{n,m}(Z_\ell, Z_r).
\]
From this and the structure of bi-non-crossing partitions, it is elementary to show that
\begin{align}
\kappa_{n,m}(Z_\ell, Z_r) = \kappa_{n+m}(Z_\ell, \ldots, Z_\ell, Z_r, \ldots, Z_r) \label{eq:bi-free-to-free-cumu}
\end{align}
where $\kappa_{n+m}$ denotes the $(n+m)^{\text{th}}$ free cumulant function.  It was also shown in \cite{CNS2015-2} that if $(Z_\ell, Z_r)$ and $(Z'_\ell, Z'_r)$ are bi-freely independent, then mixed bi-free cumulants from the pairs vanish and thus
\begin{align*}
\kappa_{n,m}&(Z_\ell + Z'_\ell, Z_r+ Z'_r) = \kappa_{n,m}(Z_\ell,   Z_r) + \kappa_{n,m}(Z'_\ell,   Z'_r).
\end{align*}

\subsection*{Additive Transformations}
\label{subsec:add-trans}

Of great use in this paper are the bi-free transformations that are used to study commuting pairs of self-adjoint operators and commuting pairs of unitary operators. Since bi-free independence subsumes free independence, the transformations used in free probability are also of use.  We begin with the setting of self-adjoint operators and additive convolution (where free or bi-free independent objects are examined).  

Let $X$ be a self-adjoint operator in a non-commutative probability space $(\A, \varphi)$ and suppose the distribution measure of $X$ with respect to $\varphi$ is $\mu_X$.  The \emph{Cauchy transform of $X$ with respect to $\varphi$} is defined by
\[
G_X(z) = \varphi((z-X)^{-1}) = \frac{1}{z} + \sum_{n\geq 1} \frac{\varphi(X^n)}{z^{n+1}} = \int_\bR \frac{1}{z-x} \, d\mu_X(x)
\]
for all $z \in \bC \setminus \bR$.  It is well known that if $\bC^+$ and $\bC^-$ denote the upper and lower half planes respectively, then $G_X(\bC^+) \subseteq \bC^-$ and $G_X(\bC^-) \subseteq \bC^+$.  Moreover, if $X$ has density $f_X$ so that $d\mu_X(x) = f_X(x) \, dx$, then $f_X$ can be recovered from $G_X$ via the formula
\begin{align}
f_X(x) = \lim_{\epsilon \searrow 0} -\frac{1}{\pi} \Im\left( G_X(x + i \epsilon)\right).  \label{eq:cauchy-inversion}
\end{align}

Recall the \emph{R-transform of $X$} is the power series that can be defined by
\[
R_X(z) = \sum_{n\geq 0} \kappa_{n+1}(X) z^n
\]
where $\kappa_{n}(X)$ is the $n^{\mathrm{th}}$ free cumulant of $X$.  Equivalently, if one defines
\[
K_X(z) = \frac{1}{z} + R_X(z),
\]
then $K_X$ is the unique function such that
\begin{align}
G_X(K_X(z)) = z = K_X(G_X(z)) \label{eq:cauchy-r-inversion}
\end{align}
for all $z \in \bC \setminus \bR$.  Thus $K_X(\bC^+) \subseteq \bC^-$ and $K_X(\bC^-) \subseteq \bC^+$.

One important use of the R-transform is that it completely determines the distribution of the free additive convolution of two self-adjoint operators via their individual distributions.  Indeed if $X$ and $X'$ are freely independent self-adjoint operators with respect to $\varphi$, then
\begin{align}
R_{X + X'}(z) = R_X(z) + R_{X'}(z). \label{eq:free-R}
\end{align}
Since $R_X$ and $R_{X'}$ are determined by $G_X$ and $G_{X'}$ respectively, and since $G_{X+X'}$ is determined by $R_{X + X'}$, we see that $G_{X+X'}$ is determined by $G_X$ and $G_{X'}$ (and thus $\mu_X$ and $\mu_{X'}$) so that $\mu_{X+X'}$ can be computed via (\ref{eq:cauchy-inversion}).

Related to computing $\mu_{X+X'}$ are the so-called subordination functions.  Indeed if
\[
\omega_X(z) = K_X(G_{X+X'}(z)) \qand \omega_{X'}(z) = K_{X'}(G_{X+X'}(z))
\]
then
\begin{align}
\omega_{X}(z) + \omega_{X'}(z) - z = \frac{1}{G_{X+X'}(z)} = \frac{1}{G_X(\omega_X(z))} = \frac{1}{G_{X'}(\omega_{X'}(z))}\label{eq:free-sub}
\end{align}
for all $z \in \bC\setminus \bR$ (see \cites{BB2007,BV1998}).

As per \cite{V2016-2}, given a pair of self-adjoint operators $(X, Y)$ in $(\A, \varphi)$, the \emph{two-variable Green's function} is defined by
\[
G_{X,Y}(z,w) = \varphi((z-X)^{-1}(w-Y)^{-1}) = \frac{1}{zw} + \sum_{\substack{n,m \geq 0 \\ n+m \geq 1}}\frac{\varphi(X^n Y^m)}{z^{n+1} w^{n+1}} 
\]
for all $(z,w) \in (\bC \setminus \bR)^2$.  If $X$ and $Y$ commute and thus have joint distribution $\mu_{X,Y}$ on $\bR^2$, then
\[
G_{X,Y}(z,w) = \int_{\bR^2} \frac{1}{z-x} \frac{1}{w-y} \, d\mu_{X,Y}(x,y).
\]
Moreover \cite{HW2016} implies that if the pair $(X, Y)$ has density $f_{X,Y}$ so that $d\mu_{X,Y}(x) = f_{X,Y}(x) \, dx\, dy$, then $f_{X,Y}$ can be recovered from $G_{X,Y}$ via the formula
\begin{align}
f_{X,Y}(x,y) = \lim_{\epsilon \searrow 0}\frac{1}{\pi^2} \Im\left(\frac{G_{X,Y}(x+i\epsilon,y + i\epsilon) - G_{X,Y}(x+i\epsilon,y - i\epsilon)   }{2i}     \right). \label{eq:bi-cauchy-inversion}
\end{align}

As in \cite{V2016-2}, the \emph{bi-free partial R-transform of $(X,Y)$} is the power series that can be defined by
\begin{align*}
R_{X,Y}(z,w) &= 	\sum_{\substack{n,m \geq 0 \\ n+m \geq 1}} \kappa_{n,m}(X, Y) z^n w^m.
\end{align*}
As proved analytically in \cite{V2016-2} and combinatorially in \cite{S2016-2}, the relationship between the bi-free partial R-transform and the two-variable Green's function is given by
\[
R_{X,Y}(z,w) = 1 + z R_X(z) + w R_Y(w) - \frac{zw}{G_{X,Y}\left( K_X(z), K_Y(w)\right)}.
\]
As such the \emph{reduced bi-free partial R-transform of $(X,Y)$} is the power series defined by
\[
\tilde{R}_{X,Y}(z,w) = \sum_{n,m \geq 1} \kappa_{n,m}(X, Y) z^n w^m.
\]
Clearly
\[
R_{X,Y}(z,w)  = z R_X(z) + w R_Y(w) + \tilde{R}_{X,Y}(z,w)
\]
so that
\begin{align}
\tilde{R}_{X,Y}(z,w)  = 1 -\frac{zw}{G_{X,Y}\left( K_X(z), K_Y(w)\right)}.\label{eq:reduced-partial-R}
\end{align}

The reduced bi-free partial R-transform enables one to determine the distribution of the bi-free additive convolution of two pairs of commuting self-adjoint operators via their individual distributions.  Indeed if $(X,Y)$ and $(X',Y')$ are pairs of commuting self-adjoint operators that are bi-freely independent with respect to $\varphi$, then $(X + X', Y+Y')$ is a pair of commuting (note $X$ and $Y'$ need not commute but will commute in distribution which is enough) self-adjoint operators and thus
\begin{align}
R_{X + X', Y+Y'}(z,w) &= R_{X,Y}(z,w) + R_{X',Y'}(z,w) \text{ and} \nonumber \\
\tilde{R}_{X + X', Y+Y'}(z,w) &= \tilde{R}_{X,Y}(z,w) + \tilde{R}_{X',Y'}(z,w). \label{eq:bi-free-reduce-R-additive}
\end{align}
As per \cite{BBGS2018}, this implies that
\begin{align*}
&\frac{zw}{G_{X+X', Y+Y'}(K_{X+X'}(z),K_{Y+Y'}(w))} +1\\
&= \frac{zw}{G_{X,Y}(K_X(z), K_Y(w))}+ \frac{zw}{G_{X',Y'}(K_{X'}(z), K_{Y'}(w))}
\end{align*}
so that
\begin{align}
\frac{1}{G_{X+X', Y+Y'}(z,w)} &+ \frac{1}{G_{X+X'}(z) G_{Y+Y'}(w)} \nonumber  \\ 
&= \frac{1}{G_{X,Y}(\omega_X(z), \omega_Y(w))} + \frac{1}{G_{X',Y'}(\omega_{X'}(z), \omega_{Y'}(w))}. \label{eq:bi-free-addivite-subordination}
\end{align}
In particular, via (\ref{eq:bi-cauchy-inversion}), it is possible to describe $\mu_{X+X', Y+Y'}$ with knowledge of the individual distributions of $(X,Y)$ and $(X',Y')$.

\subsection*{Bi-Free Multiplicative Transformation}

To examine communing pairs of unitaries $(U, V)$ in a non-commutative probability space $(\A, \varphi)$, it is first helpful to recall the free transforms for (bi-)free multiplicative convolution.  There are many approaches to the free transformations (see \cites{V1987, NS1997} for example), we will follow the structure of \cite{HW2018} for simple relation to the bi-free transformations.  Although the (bi-)free transforms generalize to non-unitary operators (see \cites{S2016-1, S2018, V2016-2} for the bi-free results), we will focus on such transforms for unitary operators only.  

Let $U$ be a unitary operator in a non-commutative probability space $(\A, \varphi)$ such that $\varphi(U) \neq 0$ and suppose $\mu_U$ is the distribution of $U$ with respect to $\varphi$.  The \emph{$\psi$-transform of $U$} is defined by
\begin{align}
\psi_U(z) = \varphi(zU (1-zU)^{-1})  = \int_\bT \frac{zs}{1-zs} \, d\mu(s) \label{eq:psi}
\end{align}
for all $z \in \bC \setminus \bT$.  

Note if $U$ has density $f_U$ so that $d\mu_U(s) = f_U(s) \, ds$ for $s \in \bT$, then $f_U$ can be recovered from $\psi_U$.  Indeed
\begin{align}
\Re(2 \psi_U(z) + 1) = \int_\bT \Re\left(\frac{1+zs}{1-zs}\right) \, d\mu(s) \label{eq:poisson}
\end{align}
is the Poisson integral of the measure $d\mu(\frac{1}{s})$ thereby yielding $\mu$ (see \cite{R1969}).

The \emph{$\eta$-transform of $U$} is defined by
\[
\eta_U(z) = \frac{\psi_U(z)}{1 + \psi_U(z)}
\]
for all $z \in \bC \setminus \bT$.  Note $\eta_U$ is holomorphic on its domain and $\eta'_U(0) = \varphi(U) \neq 0$.  Thus $\eta_U$ is invertible near the origin.  Moreover $\psi_U$ (and thus $\mu_U$) can be recovered from $\eta_U$ via the formula
\begin{align}
\psi^{-1}_U(z) = \eta^{-1}_U\left( \frac{z}{1+z}\right). \label{eq:free-S-inv-1}
\end{align}

The \emph{S-transform of $U$} is defined by
\begin{align}
S_U(z) = \frac{1}{z} \eta^{-1}_U(z).\label{eq:free-S-inv-2}
\end{align}
The importance of the S-transform is that it completely determines the distribution of the free multiplicative convolution of two untiary operators with respect to their individual distributions.  Indeed if $U$ and $U'$ are freely independent unitary operators with respect to $\varphi$, then
\begin{align}
S_{UU'}(z) = S_U(z) S_{U'}(z) = S_{U'U}(z).\label{eq:free-S-inv-3}
\end{align}
Since $\mu_U$ and $\mu_{U'}$ determine $S_U$ and $S_{U'}$ respectively, and since $S_{UU'}(z)$ determines $\eta^{-1}_{UU'}$ and thus $\psi_{UU'}$, the distribution of $UU'$ can be recovered via the above transformations and equation (\ref{eq:poisson}).

As per \cite{HW2018}, given a pair of commuting unitary operators $(U, V)$ in $(\A, \varphi)$ with $\varphi(U) \neq 0$, $\varphi(V) \neq 0$, and joint distribution $\mu_{U, V}$, the \emph{two-variable $\psi$-transformation of $(U, V)$} is given by
\begin{align*}
\psi_{U,V}(z,w) &= \varphi\left((zU)(wV)(1-zU)^{-1}(1-wV)^{-1}\right) \\
&= \int_{\bT^2} \frac{zs}{1-zs} \frac{wt}{1-wt} \, d\mu_{U,V}(s,t)
\end{align*}
for all $(z,w) \in (\bC \setminus \bT)^2$.  By \cite{HW2018}, if
\begin{align}
g_{U,V}(z,w) &= 4 \psi_{U,V}(z,w) + 2(\psi_{U}(z) + \psi_{V}(w)) + 1 \nonumber\\
&= \int_{\bT^2} \frac{1+zs}{1-zs} \frac{1+wt}{1-wt} d\mu_{U,V}(s,t) \label{eq:g-for-U-inversion}
\end{align}
for all $(z,w) \in (\bC \setminus \bT)^2$, then
\begin{align}
\Re\left( \frac{g(z,w) - g(\frac{1}{\overline{z}},w)}{2}\right) = \int_{\bT^2} \Re\left(\frac{1+zs}{1-zs}\right) \Re\left( \frac{1+wt}{1-wt}\right) \, d\mu(s,t) \label{eq:bi-free-poisson}
\end{align}
will recover the Poisson integral of the measure $d\mu\left(1/s, 1/t\right)$ and thus $\mu$  (see \cite{R1969}).  

By defining
\begin{align}
H_{U,V}(z,w) &= \psi_{U,V}(z,w) + \psi_{U}(z) +  \psi_{V}(w) + 1 \nonumber \\
&= \int_{\bT^2} \frac{1}{(1-zs)(1-wt)} \, d\mu(s,t) \label{eq:H-definition}
\end{align}
which is well-defined in a neighbourhood of $(0,0)$, the \emph{opposite bi-free partial S-transform of $(U,V)$} was defined in \cite{HW2018} to be
\begin{align}
S^{\mathrm{op}}_{U,V}(z,w) = \frac{w(z+1)}{z(w+1)} \left(\frac{H_{U,V}(\psi^{-1}_{U}(z), \psi^{-1}_{V}(w)) - (w+1)}{H_{U,V}(\psi^{-1}_{U}(z), \psi^{-1}_{V}(w)) - (z+1)}\right). \label{eq:op-S-via-H}
\end{align}
The opposite bi-free partial S-transform is analytic in a neighbourhood of $(0,0)$. 

An alternative formulation of the opposite bi-free partial S-transform was provided in \cite{S2018}*{Proposition 2.5} (also see \cite{S2018}*{Remark 2.6}).  Indeed by defining
\[
K_{U,V}(z,w) = \sum_{n,m\geq 1} \kappa_{n,m}(U, V) z^n w^m,
\]
we have
\begin{align}
S^{\mathrm{op}}_{U,V}(z,w)  = \frac{1 + \frac{1}{z} K_{U,V}(z S_U(z), w S_V(w))}{1 + \frac{1}{w} K_{U,V}(z S_U(z), w S_V(w))} \label{eq:op-S-transform-via-cummulants}
\end{align}
in a neighbourhood of $(0,0)$.

If $(U', V')$ is another pair of commuting unitary operators in $(\A, \varphi)$ with $\varphi(U') \neq 0$ and $\varphi(V') \neq 0$ that is bi-free from $(U, V)$, then \cite{HW2018} proved analytically and \cite{S2018} proved combinatorially that
\begin{align}
S^{\mathrm{op}}_{UU', V'V}(z,w) = S^{\mathrm{op}}_{U,V}(z,w)S^{\mathrm{op}}_{U', V'}(z,w). \label{eq:op-S}
\end{align}
This formula in conjunction with equation (\ref{eq:bi-free-poisson}) enables one to compute $\mu_{UU', V'V}$ in terms of $\mu_{U,V}$ and $\mu_{U', V'}$.  Indeed since $U$ and $U'$ are freely independent and $V$ and $V'$ are freely independent, it is possible to compute $\mu_{UU'}$ and $\mu_{V'V}$ as above.  As equation (\ref{eq:op-S}) enables one to compute $S^{\mathrm{op}}_{(UU', V'V)}(z,w)$ and thus $H_{U,V}(z,w)$, it is possible to compute $\psi_{U,V}$ and thus $\mu_{U,V}$ via equation (\ref{eq:bi-free-poisson}).  Thus the algorithm for determining this bi-free multiplicative convolution is more complicated then the bi-free additive convolution.

It should be pointed out that the above is considered the `opposite' bi-free partial S-transform due to historical development.  Indeed \cite{V2016-3} proved analytically and \cite{S2016-1} proved combinatorially the existence of the bi-free partial S-transform $S_{U,V}(z,w)$ such that
\[
S_{UU', VV'}(z,w) = S_{U,V}(z,w)S_{U', V'}(z,w).
\]
It was shown in \cite{HW2018} that there is a relation between these two partial S-transforms provided unitary operators are being considered with the opposite bi-free partial S-transform being more useful there and in this paper.

\section{Non-Commutative Stochastic Processes Background}
\label{sec:stochastic}

In this section, we will introduce the concepts and terminology on non-commutative stochastic processes required in the subsequent sections.  

\begin{defn}
Let $(\A, \varphi)$ be a non-commutative probability space.  A \emph{non-commutative stochastic process} is a collection $(Z_t)_{t \in T}$ of elements in $\fM$.  The index set $T$ is considered a time parameter.  A non-commutative stochastic process $(Z_t)_{t \in T}$ is said to be
\begin{itemize}
\item \emph{self-adjoint} if $Z_t$ is self-adjoint for all $t \in T$.
\item \emph{unitary} if $Z_t$ is unitary for all $t \in T$.
\end{itemize}
\end{defn}

Of particular interest in free probability are the following types of non-commutative processes.

\begin{defn}[\cite{B1998}]
A self-adjoint non-commutative stochastic process $(X_t)_{t \in T}$ in a non-commutative probability space $(\A, \varphi)$ is said to have \emph{freely additive increments} if for all $t_1 < t_2 < \cdots < t_n$ in $T$, the operators $X_{t_1}, X_{t_2} - X_{t_1}, \ldots, X_{t_n} - X_{t_n-1}$ are freely independent.
\end{defn}

\begin{defn}[\cite{B1998}]
A unitary non-commutative stochastic process $(U_t)_{t \in T}$ in a non-commutative probability space $(\A, \varphi)$ is said to have \emph{(left) multiplicatively free increments} if for all $t_1 < t_2 < \cdots < t_n$ in $T$, the operators $U_{t_1}, U_{t_2}U^{-1}_{t_1}, \ldots, U_{t_n}U^{-1}_{t_{n-1}}$ are freely independent.
\end{defn}

This paper will mainly focus on the transitioning between different operators in a non-commutative stochastic process.  In particular, given a non-commutative stochastic process $(Z_t)_{t \in T}$, we will consider $Z_\ell$ and $Z_r$ where $\ell,r \in T$ are such that $\ell \leq r$.  In the bi-free terminology, `left operators' will be used to model the past whereas `right operators' will be used to model the future.  To study such transitions, the following terminology is useful.

\begin{defn}
Let $(Z_t)_{t \in T}$ be a non-commutative stochastic process in tracial von Neumann algebra $(\fM, \tau)$.  Denote
\begin{align*}
\fM_{t]} &= W^*(\{Z_\ell \, \mid \, \ell \leq t\}) \subseteq \fM\\
\fM_{[t]} &= W^*(\{Z_t\}) \subseteq \fM\\
\fM_{[t} &= W^*(\{Z_r \, \mid \, r \geq t\}) \subseteq \fM.
\end{align*}
It is said that $(Z_t)_{t \in T}$ is a \emph{Markov process} if for all $\ell,r \in T$ with $\ell \leq r$ we have
\[
E_{\fM_{\ell]}}(A) \in \fM_{[\ell]} \quad \text{for all }A \in \fM_{[r]}.
\]
\end{defn}

\begin{defn}
Let $(Z_t)_{t \in T}$ be a self-adjoint or unitary non-commutative stochastic process in a tracial von Neumann algebra $(\fM, \tau)$.  For $t \in T$, let $\mu_t$ be the distribution of $X_t$.  Note $W^*(Z_t)$ is isomorphic to $L_\infty(\mu_t)$.

For $\ell,r \in T$ with $\ell \leq r$, an operator $K_{\ell,r} : L_\infty(\mu_r) \to L_\infty(\mu_\ell)$ determined by
\[
E_{\fM_{[\ell]}}(h(Z_r)) = (K_{\ell,r}(h)) (Z_\ell)
\]
for all $h \in L_\infty(\mu_r)$ is called a \emph{transition operator} of the process $(Z_t)_{t \in T}$.
\end{defn}

If $(Z_t)_{t \in T}$ is a Markov process then all of the transition operators exist (see \cite{BKS1997} for example) and 
\[
E_{\fM_{\ell]}}(h(Z_r)) = E_{\fM_{[\ell]}}(h(Z_r)) = (K_{\ell,r}(h)) (Z_\ell)
\]
for all $\ell \leq r$ and $h \in L_\infty(\mu_r)$.  Furthermore, \cite{B1998} showed the existence of transition operators for all self-adjoint non-commutative stochastic processes with additively free increments and for all all unitary non-commutative stochastic processes with multiplicatively free increments.

\section{Self-Adjoint Non-Commutative Stochastic Processes}
\label{sec:self-adjoint}

In this section we will examine the relation between bi-free probability theory and the transition operators of self-adjoint non-commutative stochastic processes and, more generally, the expectation of a von Neumann algebra generated by a single self-adjoint operator onto another von Neumann algebra generated by a single self-adjoint operator.  Our initial discussions will also consider unitary operators which will be examined further in the next section.

Let $(\fM, \tau)$ be a tracial von Neumann algebra and let $X_\ell, X_r \in \fM$ be self-adjoint (respectively unitary) operators.  If $\mu_\ell$ and $\mu_r$ are the distribution measures of $X_\ell$ and $X_r$ with respect to $\tau$ respectively, then $W^*(X_\ell) \cong L_\infty(\mu_\ell)$ and $W^*(X_r) \cong L_\infty(\mu_r)$.  If $E : \fM \to W^*(X_\ell) $ is the trace-preserving conditional expectation of $\fM$ onto $W^*(X_\ell)$, then, by standard von Neumann algebra theory, for all $S \in \fM$ the value of $E(S)$ is determined by the values of
\[
\tau(E(S) X_\ell^n) = \tau(SX_\ell^n) = \tau(X_\ell^n S)
\]
for all $n \in \bN \cup \{0\}$ (respectively $n \in \bZ$). Thus, the restriction of $E$ to $W^*(X_r) \cong L_\infty(\mu_r)$ yields a transition operator $K_{\ell,r} : L_\infty(\mu_r) \to L_\infty(\mu_\ell)$ determined by
\[
E(h(X_r)) = (K_{\ell,r}(h)) (X_\ell)
\]
for all $h \in L_\infty(\mu_r)$.

Consider the GNS Hilbert space $L_2(\fM, \tau)$ generated by $(\fM, \tau)$, let $\xi = 1_\fM \in L_2(\fM, \tau)$, and for $S \in \fM$ let $L(S)$ and $R(S)$ denote the left and right actions of $S$ on $L_2(\fM, \tau)$ respectively.  Note for all $n,m \in \bN \cup \{0\}$ (respectively $n,m \in \bZ$) that
\[
\tau(X_\ell^n X_r^m) = \langle L(X_\ell)^n R(X_r)^n \xi, \xi\rangle.
\]
As this vector state is positive and $L(X_\ell)$ and $R(X_r)$ commute, there exists a probability measure $\mu_{\ell,r}$ such that if $\Omega = \bR$ (respectively $\Omega = \bT$) then
\[
\int_{\Omega^2} x^n y^m \, d\mu_{\ell,r}(x,y) = \tau(X_\ell^n X_r^m)
\]
for all $n,m \in \bN \cup \{0\}$ (respectively $n \in \bZ$).

Clearly $\mu_{\ell, r}$ completely determines the transition operator $K_{\ell, r}$.  In particular, if $\mu_{\ell,r}$ is absolutely continuous with respect to the two-dimensional Lebesgue measure and thus can be written as $f_{\ell,r}(x,y) \, dx \, dy$,  then
\[
f_\ell(x) = \int_\Omega f_{\ell,r}(x,y) \, dy
\]
is the density function of $X_\ell$ and the transition operator $K_{\ell,r} : L_\infty(\mu_r) \to L_\infty(\mu_\ell)$ is obtained via
\[
(K_{\ell,r}(h))(x) = \int_\Omega h(y) k_{\ell,r}(x,dy)
\]
where
\begin{align}
k_{\ell,r}(x, dy) = \frac{f_{\ell,r}(x,y)}{f_\ell(x)}  \, dy.\label{eq:transition-op-from-2-d-distribution}
\end{align}

Using bi-free probability theory and the above, the transition operators of some common non-commutative stochastic processes become easy to describe and, interestingly, are related to the bi-free generalizations of such objects.

\begin{exam}\label{exam:bi-free-central-limit}
Let $\H$ be a real Hilbert space, let $\H_\bC$ be the complexification of $\H$, let $\F(\H_\bC)$ denoted the Fock space associated to $\H_\bC$, and let $\tau : \B(\F(\H_\bC)) \to \bC$ denote the vacuum vector state.  For an index set $T$, let $(f_t)_{t \in T}$ be a set of vectors in $\H$.  As per \cite{BKS1997}, a \emph{free (centred) Gaussian Markov process} is the self-adjoint non-commutative stochastic process $(X_t)_{t \in T}$ where
\[
X_t = l(f_t) + l^*(f_t)
\]
for all $t \in T$ where $l$ and $l^*$ are the left creation and annihilation  operators on $\F(\H_\bC)$ respectively.  Moreover \cite{BKS1997}*{Remark 3.4} implies that free Gaussian Markov processes depend only on the covariance function $c : T \times T \to \bR$ defined by
\[
c(\ell, r) = \langle f_\ell, f_r\rangle
\]
and does not depend on the choice of $\H$ nor the choice of $(f_t)_{t \in T}$.  Some examples of free Gaussian Markov processes include the free Brownian motion whose covariance function is given by $c(\ell, r) = \min(\ell, r)$ with $T = [0, \infty)$, the free Brownian bridge whose covariance function is given by $c(\ell, r) = \ell(1-r)$ for $\ell \leq r$ with $T = [0,1]$, and the free Ornstein-Uhlenbeck process whose covariance function is given by $c(\ell, r) = e^{-|\ell-r|}$ with $T = \bR$.

In \cite{BKS1997}*{Theorem 4.6} (also see \cite{B1998}*{5.3}) the transition operators for the free (centred) Gaussian Markov process were described.  Using standard bi-free probability theory, we obtain a connection between the transition operators of these free Gaussian Markov processes and the bi-free Gaussian distributions.  Indeed for $\ell, r \in T$ with $\ell < r$, note the the joint distribution of $L(X_\ell)$ and $R(X_r)$ is precisely the distribution of the pair
\[
(l(f_\ell) + l^*(f_\ell), r(f_r) + r^*(f_r))
\]
where $r$ and $r^*$ are the right creation and annihilation operators on $\F(\H_\bC)$ respectively.  This pair is a centred self-adjoint bi-free central limit pair by \cite{V2014}*{Theorem 7.6} with bi-free partial and reduced bi-free partial R-transforms given by
\begin{align*}
R_{X_\ell, X_r}(z,w) &= c(\ell, \ell) z^2 + c(\ell, r) zw + c(r,r) w^2\\
\tilde{R}_{X_\ell, X_r}(z,w) &= c(\ell, r) zw.
\end{align*}

The distribution $\mu_{\ell, r}$ was computed in \cite{HW2016}*{Example 3.4} (also see the correction comment in \cite{HW2018}*{Comment 1}) in the case that $c(\ell, \ell) = 1 = c(r,r)$ and $c(\ell, r) = c$.  For completeness and to demonstrate the technology from Section \ref{subsec:add-trans}, we will provide a proof of the formula for the density of $\mu_{\ell, r}$ in the general case.

By equation (\ref{eq:reduced-partial-R}) we know that
\[
G_{X_\ell, X_r}(K_{X_\ell}(z), K_{X_r}(w)) = \frac{zw}{1-c(\ell, r)zw }.
\]
Recall $X_\ell$ and $X_r$ are semicircular operators with spectrum $[-2\sqrt{c(\ell, \ell)} ,2\sqrt{c(\ell, \ell)}]$ and $[-2\sqrt{c(r, r)}, 2\sqrt{c(r, r)}]$ respectively.  Hence the distribution $\mu_{\ell, r}$ is supported on the Cartesian product of these two intervals.  Since
\[
G_{X_\ell}(z) = 
\frac{z - \sqrt{z^2 - 4c(\ell,\ell)}}{2 c(\ell,\ell)}  \qand G_{X_r}(w) = 
\frac{w - \sqrt{w^2 - 4c(r,r)}}{2 c(r,r)} 
\]
for $\Im(z), \Im(w) > 0$, since $G_{X_r}(\overline{w}) = \overline{G_{X_r}(w)}$, and since equation (\ref{eq:cauchy-r-inversion}) enables us to compute $G_{(X_\ell, X_r)}(z,w)$ by replacing $z$ and $w$ with $G_{X_\ell}(z)$ and $G_{X_r}(w)$ respectively, we see by equation (\ref{eq:bi-cauchy-inversion}) with $a = c(\ell,\ell)$, $b = c(r,r)$, and $c = c(\ell, r)$  that for all $(x,y)$ in the support of $\mu_{\ell, r}$,
\begin{align*}
&f_{\ell,r}(x,y) \\
&= \frac{1}{\pi^2}   \Im\left( \frac{1}{2i}\left[ \frac{\left(\frac{x- i\sqrt{4 a - x^2}}{2a} \right)\left(\frac{y- i\sqrt{4 b - y^2}}{2b}\right) }{1 - c \left(\frac{x- i\sqrt{4 a - x^2}}{2a} \right) \left(\frac{y- i\sqrt{4 b - y^2}}{2b}   \right)} -  \frac{\left(\frac{x- i\sqrt{4 a - x^2}}{2a} \right)\left(\frac{y+ i\sqrt{4 b - y^2}}{2b}\right) }{1 - c \left(\frac{x- i\sqrt{4 a - x^2}}{2a} \right) \left(\frac{y+ i\sqrt{4 b - y^2}}{2b}   \right)}   \right] \right)\\
&= \frac{1}{2\pi^2}\left(\frac{ -\frac{xy}{4ab} + \frac{\sqrt{4a-x^2}\sqrt{4b-y^2}}{4ab} + \frac{c}{ab}}{\left(1 + \frac{c^2}{ab}\right) - \frac{c}{2 a b} xy + \frac{c}{2 a b} \sqrt{4a - x^2}\sqrt{4b - y^2}} - \frac{-\frac{xy}{4ab} -\frac{\sqrt{4a-x^2}\sqrt{4b-y^2}}{4ab} + \frac{c}{ab}    }{ \left(1 + \frac{c^2}{ab}\right) - \frac{c}{2 a b} xy - \frac{c}{2 a b} \sqrt{4a - x^2}\sqrt{4b - y^2} }              \right)\\
&= \frac{1}{4\pi^2ab} \left(\frac{ \left(1 - \frac{c^2}{ab}  \right)\sqrt{4a-x^2} \sqrt{4b-y^2}  }{\left(1- \frac{c^2}{ab}\right)^2 - \frac{c}{ab}\left(1 + \frac{c^2}{ab}\right) xy + \frac{c^2}{a^2 b^2} \left( bx^2 + ay^2\right)  }    \right).
\end{align*}
To simplify this expression, note if we let $\lambda_t = \sqrt{c(t,t)}$ for $t \in \{\ell, r\}$ and
\[
\lambda_{\ell,r} = \frac{c(\ell,r)}{\lambda_\ell \lambda_r},
\]
then
\[
f_{\ell,r}(x,y) = \frac{1}{4\pi^2 \lambda_\ell^2 \lambda_r^2} \left(\frac{ \left(1 - \lambda_{\ell, r}^2 \right)\sqrt{4\lambda_\ell^2-x^2} \sqrt{4\lambda^2_r-y^2}  }{\left(1- \lambda_{\ell, r}^2\right)^2 - \lambda_{\ell, r}\left(1 + \lambda_{\ell, r}^2\right) \left( \frac{x}{\lambda_\ell}\right) \left(\frac{y}{\lambda_r}\right)  + \lambda_{\ell, r}^2 \left(\left(\frac{x}{\lambda_\ell}   \right)^2 +\left( \frac{y}{\lambda_r}  \right)^2         \right)  }    \right).
\]
Therefore, since $X_\ell$ is a centred semicircular operator with variance $c(\ell, \ell) = \lambda_\ell^2$ so that
\[
f_{\ell}(x) = \frac{1}{2\pi \lambda_\ell^2} \sqrt{4 \lambda_\ell^2 - x^2}
\]
for all $x \in [-2\lambda_\ell, 2\lambda_\ell]$,  we obtain by equation (\ref{eq:transition-op-from-2-d-distribution}) that
\[
k_{\ell,r}(x, dy) = \frac{1}{2\pi \lambda_r^2} \frac{(1-\lambda_{\ell,r}^2) \sqrt{4 \lambda_r^2 - y^2} \, dy}{(1-\lambda_{\ell,r}^2)^2 - \lambda_{\ell,r}(1 + \lambda_{\ell,r}^2) \left(\frac{x}{\lambda_\ell}\right)\left(\frac{y}{\lambda_r}\right) + \lambda_{\ell,r}^2\left(  \left(\frac{x}{\lambda_\ell}\right)^2 + \left(\frac{y}{\lambda_r}\right)^2\right)}
\]
in agreement with \cite{BKS1997}*{Theorem 4.6}.

Note that although the above focused on the free Gaussian Markov processes, the same bi-free computations work for any non-commutative stochastic process consisting of semicircular operators.
\end{exam}

\begin{exam}\label{exam:poisson}
Recall the following example from \cite{VDN1992}.  Let $(\fM, \tau)$ be a tracial von Neumann algebra and let $I \mapsto P_I$ be a projection valued process; that is, this map is normal, projection valued, if $I, J \subseteq [0,1]$ are disjoint then $P_I P_J = 0$ and $P_I + P_J = P_{I\cup J}$, and $\tau(P_I) = |I|$ for all $I \subseteq [0,1]$ where $|I|$ denotes the Lebesgue measure of $I$.  If $S$ is a centred semicircular operator in $\fM$ of variance 1 and free from $\{P_I \, \mid \, I \subseteq [0,1]\}$, then $I \mapsto SP_IS$ is called a \emph{free Poisson process}.

For all $t \in [0,1]$ let $X_t = SP_{[0,t)} S$ which yields the self-adjoint non-commutative process $(X_t)_{t \in[0,1]}$.  Then
\[
\kappa_m(X_t) = t
\]
for all $t \in [0,1]$. Furthermore, for $\ell,r \in [0,1]$ with $\ell \leq r$, we obtain that
\[
\kappa_{n+m}(X_\ell, \ldots, X_\ell, X_r, \ldots, X_r) = \begin{cases}
\ell &\text{if } n > 0 \\
r & \text{if $n=0$ and $m > 0$}
\end{cases}
\]
since $X_r = X_\ell + SP_{[\ell, r)}S$ and since $SP_{[\ell, r)}S$ is freely independent from $X_\ell$ by \cite{NS1996}*{Theorem 1.6}.  Thus, equation (\ref{eq:bi-free-to-free-cumu}) implies that the ordered bi-free cumulants for the bi-free compound Poisson distribution with rate $\lambda = r$ and jump size $\nu = \frac{\ell}{r} \delta_{(1,0)} + \frac{r-\ell}{r}\delta_{(1,1)}$ where $\delta_{(x,y)}$ is the Kronecker delta measure at $(x,y)$ (see \cite{GHM2016}).   Thus the distribution of $(L(X_\ell), R(X_r))$ is the bi-free compound Poisson distribution and thus the transition operator from $X_r$ to $X_\ell$ can be computed.  Note by \cite{HHW2018}*{Corollary 7.6} that such distributions are not full. 
\end{exam}
 
Note the free Poisson process is an example of a self-adjoint non-commutative stochastic process with freely additive increments.  Using technology from bi-free probability, an approach to study such processes that differs from that in \cite{B1998}*{Theorem 3.1} is obtained:

\begin{thm}\label{thm:free-additive-increments-Cauchy-transform-formula}
Let $X$ and $Y$ be freely independent self-adjoint operators in a tracial von Neumann algebra $(\fM, \tau)$.  Then
\[
G_{L(X), R(X+Y)}(z,w) = - \frac{G_X(z) - G_{X+Y}(w)}{z - K_X(G_{X+Y}(w))}.
\]
Therefore, since equation (\ref{eq:free-sub}) allows one to compute $G_{X+Y}(w)$ based on the distributions of $X$ and $Y$, equation (\ref{eq:bi-cauchy-inversion}) enables one to compute the distribution of the pair $(L(X), R(X+Y))$ and thus the transition operator of $X+Y$ onto $X$.

In particular if $(X_t)_{t \in T}$ is a self-adjoint non-commutative stochastic process with freely additive increments, the above holds for $X = X_\ell$ and $Y = X_r - X_\ell$ for all $\ell < r$ and thus $G_{L(X_\ell), R(X_r)}$ is computable via the individual transformations of $X_\ell$ and $X_r$.
\end{thm}
\begin{proof}
Since $X$ and $Y$ are freely independent, Theorem \ref{thm:bi-free-from-free} implies that the pairs
\[
(L(X), R(X)) \qand (L(0), R(Y))
\]
are bi-freely independent with respect to $\tau$.   Hence equation (\ref{eq:bi-free-reduce-R-additive}) along with bi-free cumulants of constants vanishing (see \cite{CNS2015-1}*{Proposition 6.4.1}) imply that
\begin{align*}
\widetilde{R}_{L(X), R(X+Y)}(z,w) &= \widetilde{R}_{L(X), R(X)}(z,w) + \widetilde{R}_{L(0), R(Y)}(z,w) \\
&= \widetilde{R}_{L(X), R(X)}(z,w) + 0 \\
&= \sum_{n,m\geq 1} \kappa_{n,m}(L(X), R(X)) z^n w^m \\
&= \sum_{n,m\geq 1} \kappa_{n+m}(X) z^n w^m \\
&= \sum_{k \geq 2} \kappa_k(X)  zw\frac{z^{k-1} - w^{k-1}}{z-w} \\
&= zw \sum_{k \geq 1} \kappa_{k+1}(X) \frac{z^{k} - w^{k}}{z-w} \\
&= zw \frac{R_{X}(z) - R_{X}(w)}{z-w}.
\end{align*}
Alternatively, since $X$ and $Y$ are freely independent, for $n \geq 1$ and $m \geq 1$, 
\begin{align*}
\kappa_{n+m}(\underbrace{X, \ldots, X}_{n \text{ times}}, \underbrace{X+Y, \ldots X+Y}_{m \text{ times}}) &= \kappa_{n+m}(X, \ldots, X) + \Sigma
\end{align*}
where $\Sigma$ is a sum of free cumulants containing $n$ copies of $X$ and $m$ operators that are either $X$ or $Y$ with not all being $X$, and thus vanishes as $X$ and $Y$ are freely independent.  Therefore, by equation (\ref{eq:bi-free-to-free-cumu}), 
\[
\kappa_{n,m}(X,X+Y) = \begin{cases}
\kappa_{n+m}(X) & \text{if }n > 0 \\
\kappa_{n+m}(Y) & \text{if }n = 0
\end{cases}
\]
and the above reduced bi-free partial R-transform can be computed from this.  

Using equation (\ref{eq:reduced-partial-R}) we obtain that
\begin{align*}
-\frac{zw}{G_{L(X),R(X+Y)}\left( K_{X}(z), K_{X+Y}(w)\right)} &= zw \frac{R_{X}(z) - R_{X}(w)}{z-w} - 1 \\
&= zw\frac{(R_{X}(z) + \frac{1}{z} ) -   (R_{X}(w) + \frac{1}{w})}{z-w}
\end{align*}
so
\[
G_{L(X),R(X+Y)}\left( K_{X}(z), K_{X+Y}(w)\right) = - \frac{z-w}{K_{X}(z) -   K_X(w)}.
\]
Therefore equation (\ref{eq:cauchy-r-inversion}) yields the desired equation.
\end{proof}

Similar to  \cite{B1998}*{Theorem 3.1}, Theorem \ref{thm:free-additive-increments-Cauchy-transform-formula} can be used to compute the transition operators of many self-adjoint non-commutative stochastic processes:
\begin{exam}
The \emph{free Cauchy process} is the self-adjoint non-commutative stochastic process with freely additive increments where
\[
\mu_t(dx) = \frac{1}{\pi} \frac{t}{x^2 + t^2} \, dx
\]
for all $t \in [0, \infty)$. To compute the transition operators, note for all $t \in [0, \infty)$ that
\[
G_{X_t}(z) = \begin{cases}
\frac{1}{z + it} & \text{if }\Im(z) > 0 \\
\frac{1}{z-it} & \text{if } \Im(z) < 0
\end{cases}
\]
so that 
\[
K_{X_t}(z) = \begin{cases}
\frac{1}{z} + it & \text{if }\Im(z) > 0 \\
\frac{1}{z}-it & \text{if } \Im(z) < 0
\end{cases}.
\]
Thus if $l,r \in [0,\infty)$ are such that $l < r$, then for all $z,w \in \bC$ with $\Im(z) > 0$ and $\Im(w) > 0$ we have by Theorem \ref{thm:free-additive-increments-Cauchy-transform-formula} that
\begin{align*}
G_{L(X_\ell), R(X_r)}(z,w) &= - \frac{\frac{1}{z+i\ell} - \frac{1}{w+ir}}{z - \left(\frac{1}{\frac{1}{w + ir}} + i\ell\right)} = \frac{1}{z+i\ell}\frac{1}{w + ir},
\end{align*}
and for all $z,w \in \bC$ with $\Im(z) > 0$ and $\Im(w) < 0$ we have by Theorem \ref{thm:free-additive-increments-Cauchy-transform-formula} that
\begin{align*}
G_{L(X_\ell), R(X_r)}(z,w) &= - \frac{\frac{1}{z+i\ell} - \frac{1}{w-ir}}{z - \left(\frac{1}{\frac{1}{w - ir}} - i\ell\right)} \\
&= \frac{(z-w) + i(\ell + r)}{(z+i\ell)(w-ir)((z-w) + i(r-\ell))}.
\end{align*}
Hence equation (\ref{eq:bi-cauchy-inversion}) implies that 
\begin{align*}
f_{\ell,r}(x,y)
&= \frac{1}{\pi^2}   \Im\left(\frac{1}{2i} \frac{1}{x + i\ell} \left[\frac{1}{y + i r} - \frac{(x - y) + i (\ell+r)}{(y - ir)((x-y) + i (r-\ell)}  \right]     \right)\\
&= \frac{1}{\pi^2} \frac{\ell}{x^2 + \ell^2} \frac{r-\ell}{(x-y)^2 + (r-\ell)^2}.
\end{align*}
Therefore, since
\[
f_\ell(x) = \frac{1}{\pi} \frac{\ell}{x^2 + \ell^2},
\]
we obtain by equation (\ref{eq:transition-op-from-2-d-distribution}) that
\[
k_{\ell,r}(x, dy) = \frac{f_{\ell,r}(x,y)}{f_\ell(x)} \, dy = \frac{1}{\pi} \frac{r-\ell}{(x-y)^2 + (r-\ell)^2} dy
\]
in agreement with \cite{B1998}*{5.1}.
\end{exam}

One benefit of the bi-free approach to non-commutative stochastic processes over \cite{B1998} is the ability to understand a free additive convolution of non-commutative stochastic processes without freely additive increments.  This is accomplished via the following.

\begin{thm}\label{thm:free-convolution-self-adjoint-stochastic-processes}
Let $X_1, X_2, Y_1, Y_2$ be self-adjoint operators in a tracial von Neumann algebra $(\fM, \tau)$ such that $\alg(\{X_1, Y_1\})$ and $\alg(\{X_2, Y_2\})$ are freely independent.  Thus $G_{X_1 + X_2}(z)$ and $G_{Y_1 + Y_2}(w)$ can be computed via the process described after equation (\ref{eq:free-R}).  With
\[
\omega_{X_k}(z) = K_{X_k}(G_{X_1 + X_2}(z)) \qand \omega_{Y_k}(w) = K_{Y_k}(G_{Y_1 + Y_2}(w))
\]
for $k = 1,2$, we have
\begin{align*}
\frac{1}{G_{X_1+X_2, Y_1+Y_2}(z,w)} &+ \frac{1}{G_{X_1+X_2}(z) G_{Y_1+Y_2}(w)} \\
&= \frac{1}{G_{X_1,Y_1}(\omega_{X_1}(z), \omega_{Y_1}(w))} + \frac{1}{G_{X_2,Y_2}(\omega_{X_2}(z), \omega_{Y_2}(w))}.
\end{align*}
Thus equation (\ref{eq:bi-cauchy-inversion}) can be used to compute the transition operator of $Y_1 + Y_2$ onto $X_1 + X_2$.
\end{thm}
\begin{proof}
Since $\alg(\{X_1, Y_1\})$ and $\alg(\{X_2, Y_2\})$ are freely independent, Theorem \ref{thm:bi-free-from-free} implies that
\begin{align*}
&(\alg(\{L(X_1), L(Y_1)\}), \alg(\{R(X_1), R(Y_1)\})) \qand \\
& (\alg(\{L(X_2), L(Y_2)\}), \alg(\{R(X_2), R(Y_2)\})) 
\end{align*}
are bi-freely independent with respect to $\tau$.   Hence
\[
(L(X_1),  R(Y_1)) \qand ( L(X_2) ,  R(Y_2))
\]
are bi-freely independent with respect to $\tau$.  Thus $(L(X_1 + X_2), R(Y_1, Y_2))$ is the additive bi-free convolution of $(L(X_1),  R(Y_1))$ and $( L(X_2) ,  R(Y_2))$ so equation (\ref{eq:bi-free-addivite-subordination}) applies thereby yielding the result.
\end{proof}

Theorem \ref{thm:free-convolution-self-adjoint-stochastic-processes} is particularly useful to compute the transition operators of a free additive convolution of two self-adjoint non-commutative stochastic processes where the transition operators are known, even if those non-commutative stochastic processes do not have freely additive increments.  Indeed suppose $(X_t)_{t \in T}$ is a self-adjoint non-commutative stochastic process in $(\fM, \tau_\fM)$ and $(Y_t)_{t \in T}$  is a self-adjoint non-commutative stochastic process in $(\fN, \tau_\fN)$.  By consider the reduced free product $(\fM \ast \fN, \tau_\fM \ast \tau_\fN)$, we can consider $(X_t)_{t \in T}$ and $(Y_t)_{t \in T}$ in the same tracial von Neumann algebra in such a way that $X_t$ and $Y_s$ are freely independent for all $t,s \in T$. The self-adjoint non-commutative stochastic process $(X_t + Y_t)_{t \in T}$ then falls under the purview of Theorem \ref{thm:free-convolution-self-adjoint-stochastic-processes} to compute the transition operators.  

As examples of the above are at least as complicated as computing the distributions of free additive convolutions, which often need computer assistance, we only provide the following example.

\begin{exam}
In the above exposition, suppose $X_t$ and $Y_t$ are centred semicircular operators for all $t \in T$.  For $\ell, r \in T$ with $\ell < r$, Example \ref{exam:bi-free-central-limit} yields the transition operator of $X_r$ onto $X_\ell$ (respectively $Y_r$ onto $Y_\ell$) where $c(t,t) = \tau_\fM(X_t^2)$ (respectively $c(t,t) = \tau_\fN(Y_t^2)$) for $t \in \{\ell, r\}$ and $c(\ell, r) = \tau_\fM(X_\ell X_r)$ (respectively $c(\ell, r) = \tau_\fN(Y_\ell Y_r)$).  As the (bi-)free additive convolution of such operators is of the same form, the transition operator of $X_r + Y_r$ onto $X_\ell + Y_\ell$ is given by the transition operators in Example \ref{exam:bi-free-central-limit} with $c(t,t) = \tau_\fM(X_t^2) + \tau_\fN(Y_t^2)$ for $t \in \{\ell, r\}$ and $c(\ell, r) = \tau_\fM(X_\ell X_r) + \tau_\fN(Y_\ell Y_r)$.
\end{exam}

Furthermore, the algebraic bi-free central limit theorem (\cite{V2014}*{Theorem 7.9}) immediately yield the following for limits of free additive convolutions of non-commutative stochastic processes.  In particular, the transition operators of such non-commutative stochastic processes `tend' to those from Example \ref{exam:bi-free-central-limit}.

\begin{thm}\label{thm:central}
Let $\{(X_{n,t})_{t \in T}\}_{n \in \bN}$ be a sequence of self-adjoint non-commutative stochastic processes in a tracial von Neumann algebra $(\fM, \tau)$ such that for all $\ell, r \in T$ with $\ell < r$, $\alg(\{X_{n,\ell}\}_{n \in \bN})$ and $\alg(\{X_{n,r}\}_{n \in \bN})$ are freely independent.  Suppose further that
\begin{enumerate}
\item $\tau(X_{n,t}) = 0$ for all $n \in \bN$ and $t \in T$,
\item $\sup_{n \in \bN} \left|\tau(X_{n,t_1} \cdots X_{n, t_k})\right| < \infty$ for all $\ell, r \in T$, for all $k \in \bN$, and for all $t_1, \ldots, t_k \in \{\ell, r\}$, and
\item $\lim_{N \to \infty} \frac{1}{N} \sum^N_{n=1} \tau(X_{n,\ell} X_{n,r}) = c(\ell, r)$ for all $\ell, r \in T$.
\end{enumerate}
By letting $S_{N, t} = \frac{1}{N} \sum^N_{n=1} X_{n,t}$ for all $t \in T$ and $N \in \bN$ and letting $E_{N,t} : \fM \to W^*(S_{N, t})$ denote the conditional expectation onto $W^*(S_{N, t})$, we have for all $\ell, r \in T$ with $\ell < r$ and $m \in \bN$ that
\[
\lim_{N \to \infty} \tau(E_{N,\ell}(S_{N,r}^m)) = \int_{\bR^2} y^m f_{\ell, r}(x,y) \, dx \, dy
\]
where $f_{\ell, r}$ is as in Example \ref{exam:bi-free-central-limit}.
\end{thm}

Of course, a stronger analytic version of the bi-free central limit theorem could immediately yield stronger information about the transition operators of such processes.

\section{Unitary Stochastic Processes and Bi-Free Probability}
\label{sec:unitary}

The same ideas used in the previous section  can be transferred to describe the transition operators for unitary operators in von Neumann algebra using bi-free probability theory and the S-transforms.  We begin with the following.

\begin{thm}\label{thm:free-multiplicative-increments-Cauchy-transform-formula}
Let $U$ and $V$ be freely independent unitary operators in a tracial von Neumann algebra $(\fM, \tau)$ such that $\tau(U) \neq 0$ and $\tau(V) \neq 0$.  Then
\[
H_{L(U),R(VU)}(z,w) = 1 + \frac{z \psi_U(z) - \psi^{-1}_U(\psi_{VU}(w)) \psi_{VU}(w)}{z-  \psi^{-1}_U(\psi_{VU}(w))}.
\]
Therefore, since equations (\ref{eq:free-S-inv-1}), (\ref{eq:free-S-inv-2}), and (\ref{eq:free-S-inv-3}) allows one to compute $\psi_{VU}$ based on the distributions of $U$ and $V$, it is possible to compute $\psi_{L(U), R(VU)}$, $\psi_U$, and $\psi_{VU}$.  Thus $g_{L(U),R(VU)}$ is computable by equation (\ref{eq:g-for-U-inversion}) so the distribution of $(L(U), R(VU))$ and the transition operator of $VU$ onto $U$ are computable via equation (\ref{eq:poisson}).

In particular if $(U_t)_{t \in T}$ is a unitary non-commutative stochastic process with (left) multiplicatively free increments, the above holds for $U = U_\ell$ and $V = U_r U_\ell^{-1}$ for all $\ell < r$  and thus $H_{L(U_\ell), R(U_r)}$ is computable via the individual transformations of $U_\ell$ and $U_r$.
\end{thm}
\begin{proof}
Since $U$ and $V$ are freely independent, Theorem \ref{thm:bi-free-from-free} implies that the pairs
\[
(L(U), R(U)) \qand ( L(V),  R(V))
\]
are bi-freely independent with respect to $\tau$.  Hence equation (\ref{eq:op-S}) implies that 
\begin{align*}
S^{\mathrm{op}}_{L(U),R(VU)}(z,w) &= S^{\mathrm{op}}_{L(1)L(U), R(U)R(V)}(z,w)  \\
&= S^{\mathrm{op}}_{L(1), R(V)}(z,w)  S^{\mathrm{op}}_{L(U), R(U)}(z,w) 
\end{align*}
However, by equation (\ref{eq:op-S-transform-via-cummulants}) together with the fact that  bi-free cumulants of constants vanishing (see \cite{CNS2015-1}*{Proposition 6.4.1}), we obtain that
\[
S^{\mathrm{op}}_{L(1), R(V)}(z,w)  = 1
\]
so that
\begin{align*}
S^{\mathrm{op}}_{L(U),R(VU)}(z,w) &= S^{\mathrm{op}}_{L(U), R(U)}(z,w).
\end{align*}
Hence equation (\ref{eq:op-S-via-H}) implies that
\[
H_{L(U),R(VU)}(\psi^{-1}_U(z), \psi^{-1}_{VU}(w)) = H_{L(U),R(U)}(\psi^{-1}_U(z), \psi^{-1}_{U}(w))
\]
and thus
\[
H_{L(U),R(VU)}(z,w) = H_{L(U),R(U)}(z, \psi^{-1}_{U}(\psi_{VU}(w))).
\]
Note if $\mu_U$ is the distribution of $U$, then distribution of $(L(U), R(U))$ is supported on
\[
D = \{(s,s) \, \mid \, s \in \bT\}
\]
with $d\mu_{L(U),R(U))}(s,s) = d \mu_U(s)$.  Hence equations (\ref{eq:psi}) and (\ref{eq:H-definition}) imply that
\begin{align*}
H_{(U,U)}(z,w) &= \int_{\bT} \frac{1}{(1-zs)(1-ws)} d\mu_U(s) \\
&= \int_\bT 1 + \frac{z^2s}{(z-w)(1-zs)} - \frac{w^2s}{(z-w)(1-ws)} \, d\mu_U(s) \\
&= 1 + \frac{z \psi_U(z) - w \psi_U(w)}{z-w}
\end{align*}
so the result follows.
\end{proof}

\begin{exam}
In \cite{B1998}*{5.2}, a unitary multiplicative analogue of the free Levy process $(U_t)_{t \geq 0}$ was examined where
\[
\mu_t(ds) = \frac{1 - e^{-2t}}{|s- e^{-t}|^2}  \, ds
\]
and thus
\[
\psi_{U_t}(z) = \frac{z}{e^t-z}
\]
for all $z \in \bC$ with $|z| < 1$.  Thus if $l,r \in [0,\infty)$ are such that $l < r$, then for all $z,w \in \bC \setminus \bT$ we have by Theorem \ref{thm:free-multiplicative-increments-Cauchy-transform-formula} that
\begin{align*}
H_{L(U_\ell), R(U_r)}(z,w) &= 1 + \frac{z \psi_{U_\ell}(z) - e^{\ell-r} w \psi_{U_r}(w)}{z-e^{\ell-r} w}.
\end{align*}
Note that
\[
\psi_{U_\ell}\left(\frac{1}{\overline{z}}\right) = -1 - \overline{\psi_{U_\ell}(z)}.
\]
Therefore, since
\[
g_{L(U_\ell), R(U_r)}(z,w) = 4 H_{L(U_\ell), R(U_r)}(z,w) - 2 \psi_{U_\ell}(z) - 2 \psi_{U_r}(w) - 3,
\]
equation (\ref{eq:bi-free-poisson}) implies for all $s,t \in \bT$ that the density of $d\mu\left(\frac{1}{s}, \frac{1}{t}\right)$ is given by
\begin{align*}
&2 \Re\left(\frac{s \psi_{U_\ell}(s) - e^{\ell-r} t \psi_{U_r}(t)}{s - e^{\ell-r} t} - \frac{-s -s \overline{\psi_{U_\ell}(s)}  - e^{\ell-r}  t \psi_{U_r}\left(t\right)}{s - e^{\ell-r} t} \right) - \Re\left( \psi_{U_\ell}(s)+ \overline{\psi_{U_\ell}(s)} + 1\right)\\
&= \Re\left(\frac{2s}{s-e^{\ell-r} t}     -   1 \right) \left( \psi_{U_\ell}(s)+ \overline{\psi_{U_\ell}(s)} + 1\right)\\
&= \Re\left(\frac{(s+e^{\ell-r}t)(\overline{s} - e^{\ell-r} \overline{t}) }{|s-e^{\ell-r} t|^2}\right) \Re\left( 2\psi_{U_\ell}(s)+ 1\right)\\
&= \frac{1-e^{2(\ell-r)}}{|s - e^{\ell-r} t|^2} \frac{e^{2\ell} - 1}{|e^\ell-s|^2}.
\end{align*}
Thus, by density of $d\mu(s,t)$ is given by
\begin{align*}
\frac{1-e^{2(\ell-r)}}{|\frac{1}{s} - e^{\ell-r} \frac{1}{t}|^2} \frac{e^{2\ell} - 1}{|e^\ell-\frac{1}{s}|^2}&=\frac{1-e^{2(\ell-r)}}{|s^{-1}t - e^{\ell-r} |^2} \frac{1-e^{-2\ell} }{|e^{-\ell}-s|^2}.
\end{align*}
Thus, by dividing by the density of $\mu_\ell$, we obtain that
\[
k_{\ell, r}(s,dt) = \frac{1-e^{2(\ell-r)}}{|s^{-1}t - e^{\ell-r} |^2}  dt,
\]
which agrees with \cite{B1998}*{5.2}.
\end{exam}

Again one benefit of the bi-free approach to non-commutative stochastic processes over \cite{B1998} is the ability to understand a free multiplicative convolution of non-commutative stochastic processes without freely multiplicative increments.  This is accomplished via the following.

\begin{thm}\label{thm:free-convolution-unitary-stochastic-processes}
Let $U_1, U_2, V_1, V_2$ be unitary operators with non-vanishing first moments in a tracial von Neumann algebra $(\fM, \tau)$ such that $*\text{-}\alg(\{U_1, V_1\})$ and $*\text{-}\alg(\{U_2, V_2\})$ are freely independent so $\psi_{U_1 U_2}(z)$ and $\psi_{V_1V_2}(w)$ can be computed via the process described after equation (\ref{eq:free-S-inv-3}).  The pair 
\[
(L(U_1U_2), R(V_1V_2)) = (L(U_1)L(U_2), R(V_2)R(V_1))
\]
is the opposite bi-free multiplicative convolution of $(L(U_1), R(V_1))$ and $(L(U_2), R(V_2))$.  Therefore $\psi_{U_1 U_2, V_1 V_2}(z,w)$ can be computed via the process described after equation (\ref{eq:op-S}).  Hence equations (\ref{eq:bi-free-poisson}) and (\ref{eq:transition-op-from-2-d-distribution}) can be used to compute the transition operator of $V_1V_2$ onto $U_1U_2$.
\end{thm}
\begin{proof}
Since  $\{U_1, V_1\}$ and $\{U_2, V_2\}$ are freely independent, Theorem \ref{thm:bi-free-from-free} implies that
\begin{align*}
&(*\text{-}\alg(\{L(U_1), L(V_1)\}), *\text{-}\alg(\{R(U_1), R(V_1)\})) \qand \\
& (*\text{-}\alg(\{L(U_2), L(V_2)\}), *\text{-}\alg(\{R(U_2), R(V_2)\})) 
\end{align*}
are bi-freely independent with respect to $\tau$.   Hence
\[
( L(U_1),   R(V_1)) \qand ( L(U_2),   R(V_2)) 
\]
are bi-freely independent with respect to $\tau$.  Thus 
\[
(L(U_1U_2), R(V_1V_2)) = (L(U_1)L(U_2), R(V_2)R(V_1))
\]
is the opposite bi-free multiplicative convolution of $(L(U_1), R(V_1))$ and $(L(U_2), R(V_2))$ thereby yielding the result.
\end{proof}


\begin{thebibliography}{13}




\bib{BB2007}{article}{
    author={Belinschi, S.T.},
    author={Bercovici, H.},
    title={A new approach to subordination results in free probability},
    journal={J. Anal. Math.},
    volume={101},
    number={1},
    year={2007},
    pages={357-365}
}



\bib{BBGS2018}{article}{
    author={Belinschi, S.T.},
    author={Bercovici, H.},
    author={Gu, Y.},
    author={Skoufranis, P.},
    title={Analytic Subordination for Bi-Free Convolution},
    journal={J. Funct. Anal.},
    year={2018},
    volume={275},
    number={4},
    pages={926-966}
}



\bib{BV1998}{article}{
    author={Bercovici, H.},
    author={Voiculescu, D.},
    title={Regularity questions for free convolution},
    journal={Oper. Theory Adv. Appl.},
    volume={104},
    date={1998},
    pages={37-47}
}




\bib{B1998}{article}{
    author={Biane, P.},
    title={Processes with free increments},
    journal={Math. Z.},
    volume={227},
    number={1},
    year={1998},
    pages={143-174}
}




\bib{BKS1997}{article}{
    author={Bo\.{z}ejko, M.},
    author={K\"{u}mmerer, B.},
    author={Speicher, R.},
    title={$q$-Gaussian Processes: Non-commutative and Classical Aspects},
    journal={Commun. Math. Phys.},
    volume={185},
    year={1997},
    pages={129-154}
}



\bib{CNS2015-1}{article}{
    author={Charlesworth, I.},
    author={Nelson, B.},
    author={Skoufranis, P.},
    title={Combinatorics of Bi-Free Probability with Amalgamation},
	journal={Comm. Math. Phys.},
	volume={338},
	number={2},
    year={2015},
    pages={801-847}
}


\bib{CNS2015-2}{article}{
    author={Charlesworth, I.},
    author={Nelson, B.},
    author={Skoufranis, P.},
    title={On two-faced families of non-commutative random variables},
    journal={Canad. J. Math.},
    volume={67},
    number={6},
    year={2015},
    pages={1290-1325}
}






\bib{GHM2016}{article}{
    author={Gu, Y.},
    author={Huang, H.W.},
    author={Mingo, J.},
    title={An analogue of the L\'{e}vy-Hin\u{c}in formula for bi-free infinitely divisible distributions},
    journal={Indiana Univ. Math. J.},
    volume={65},
    number={5},
    year={2016},
    pages={1795-1831}
}
 

 
\bib{HHW2018}{article}{
    author={Hasebe, T.},
    author={Huang, H.W.},
    author={Wang, J.C.},
    title={Limit theorems in bi-free probability theory},
     journal={Probab. Theory Relat. Fields},
    volume={173},
    number={3},
    date={2018},
    pages={1081-1119}
}

 

\bib{HW2016}{article}{
    author={Huang, H.W.},
    author={Wang, J.C.},
    title={Analytic Aspects of Bi-Free Partial $R$-Transforms},
     journal={J. Funct. Anal.},
    volume={271},
    number={4},
    date={2016},
    pages={922-957}
}
 
 

\bib{HW2018}{article}{
    author={Huang, H.W.},
    author={Wang, J.C.},
    title={Harmonic analysis for the bi-free partial S-transform},
     journal={J. Funct. Anal.},
    volume={274},
    number={5},
    date={2018},
    pages={1306-1344}
}
 
  
\bib{NS1996}{article}{
    author={Nica, A.},
    author={Speicher, R.},
    title={On the multiplication of free N-tuples of noncommutative random variables},
    journal={Amer. J. Math.},
    year={1996},
    pages={799-837}
} 
 
 
 
\bib{NS1997}{article}{
    author={Nica, A.},
    author={Speicher, R.},
    title={A ``Fourier Transform" for Multiplicative Functions on Non-Crossing Partitions},
    journal={J. Algebraic Combin.},
    volume={6},
    year={1997},
    pages={141-160}
}




\bib{R1969}{book}{
    author={Rudin, W.},
    title={Function theory in polydiscs},
    publisher={WA Benjamin},
    volume={41},
    year={1969}
}



\bib{S2016-1}{article}{
    author={Skoufranis, P.},
    title={A Combinatorial Approach to Voiculescu's Bi-Free Partial Transforms},
    journal={Pacific J. Math.},
    number={2},
    volume={283},
    year={2016},
    pages={419-447}
}
 

\bib{S2016-2}{article}{
    author={Skoufranis, P.},
    title={Independences and Partial $R$-Transforms in Bi-Free Probability},
    journal={Ann. Inst. Henri Poincar\'{e} Probab. Stat.},
    volume={52},
    number={3},
    pages={1437-1473},
    year={2016}
}

\bib{S2018}{article}{
    author={Skoufranis, P.},
    title={A Combinatorial Approach to the Other Bi-Free Partial S-Transform},
    journal={Oper. Matrices.},
    volume={12},
    number={2},
    pages={333-355},
    year={2018}
}
 
 
 
\bib{S1994}{article}{
    author={Speicher, R.},
    title={Multiplicative functions on the lattice of non-crossing partitions and free convolution},
    journal={Math. Ann.},
    volume={298},
    number={1},
    year={1994},
    pages={611-628}
}
 
 
 

\bib{V1987}{article}{
    author={Voiculescu, D.},
    title={Multiplication of certain non-commuting random variables},
    journal={J. Operator Theory},
    volume={18},
    year={1987},
    pages={223-235}
}

 
 

\bib{V2014}{article}{
	author={Voiculescu, D. V.},
	title={Free probability for pairs of faces I},
	journal={Comm. Math. Phys.},
	year={2014},
	volume={332},
	pages={955-980}
}

 
 

\bib{V2016-2}{article}{
    author={Voiculescu, D. V.},
    title={Free Probability for Pairs of Faces II: 2-Variable Bi-Free Partial R-Transform and Systems with Rank $\leq 1$ Commutation},
    journal={Ann. Inst. Henri Poincar\'{e} Probab. Stat.},
    year={2016},
    volume={52},
    number={1},
    pages={1-15}
}

 





\bib{V2016-3}{article}{
    author={Voiculescu, D. V.},
    title={Free probability for pairs of faces III: 2-variable bi-Free partial $S$-transform and $T$-transforms},
    journal={J. Funct. Anal.},
    year={2016},
    volume={270},
    number={10},
    pages={3623-2628}
}




\bib{VDN1992}{book}{
	author={Voiculescu, D. V.},
	author={Dykema, K. J.},
	author={Nica, A.},
	title={Free random variables},
	publisher={American Mathematical Soc},
	volume={1},
	series={CRM Monograph},
	place={Providence, Rhode Island},
	year={1992}
}



\end{thebibliography}
\end{document}